\documentclass[11pt,a4paper,oneside,onecolumn]{article}
\usepackage{graphicx,amssymb,amsmath,theorem}
\usepackage{colortbl}
\usepackage{multicol}
\usepackage[ansinew]{inputenc}
\usepackage{lineno}
\usepackage{colortbl}

\usepackage[letterpaper]{geometry}
\newtheorem{theorem}{Theorem}[section]

\newtheorem{claim}[theorem]{Claim}

\newtheorem{corollary}[theorem]{Corollary}

\newtheorem{lemma}[theorem]{Lemma}

\def\QED{\ensuremath{{\square}}}
\def\markatright#1{\leavevmode\unskip\nobreak\quad\hspace*{\fill}{#1}}

\newenvironment{proof}
{\begin{trivlist}\item[\hskip\labelsep{\bf Proof.}]}
{\markatright{\QED}\end{trivlist}}

\def\R{\mathbb{R}}

\newtheorem{theorem*}{Theorem}
\def\R{\mathop{I\!\!R}\nolimits}

\begin{document}


\title{Extremal antipodal polygons and polytopes}

\author{O. Aichholzer\thanks{Institute for Software Technology, University of
Technology Graz, Austria {\tt oaich@ist.tugraz.at}. Partially supported by the ESF EUROCORES
programme EuroGIGA - ComPoSe, Austrian Science Fund (FWF): I~648-N18.} \and
L.E. Caraballo\thanks{Facultad de Matem\'aticas y Computaci\'on. Universidad de La Habana. {\tt luis.caraballo@iris.uh.cu}} \and
J. M. D\'{\i}az-B\'{a}\~{n}ez\thanks{Departamento de Matem\'{a}tica
Aplicada II, Universidad de Sevilla, Spain. Partially supported by
projects FEDER  P09-TIC-4840 and MEC MTM2009-08652. {\tt
\{dbanez\}@us.es.}} \and
R.~Fabila-Monroy\thanks{Departamento de
Matem\'aticas. Centro de Investigaci\'on y de Estudios Avanzados del Instituto Polit\'ecnico Nacional, Mexico City, Mexico. Partially supported by Conacyt of Mexico, grant 153984.
 {\tt ruyfabila@math.cinvestav.edu.mx}}\and
C. Ochoa\thanks{Facultad de Matem\'aticas y Computaci\'on. Universidad de La Habana. {\tt ochoa@matcom.uh.cu}}\and
P. Nigsch\thanks{University of Technology Graz, Austria {\tt paul.nigsch@student.tugraz.at}}
}

%
%
%
%
%

\date{\today}

\maketitle

\begin{abstract}
Let $S$ be a set of $2n$ points on a circle such that for each point $p \in S$ also its antipodal (mirrored with respect to the circle center) point $p'$ belongs to $S$. A polygon $P$ of size $n$ is called \emph{antipodal} if it consists of precisely one point of each antipodal pair $(p,p')$ of $S$.

We provide a complete characterization of antipodal polygons which maximize (minimize, respectively) the area among all antipodal polygons of $S$. Based on this characterization, a simple linear time algorithm is presented for computing extremal antipodal polygons. Moreover, for the generalization of antipodal polygons to higher dimensions we show that a similar characterization does not exist.
\end{abstract}


\section{Introduction}

For a point $p=(x_1,x_2) \in \R^2$, let $p':=(-x_1,-x_2)$ be the antipodal point of $p$. Consider a set $S$ of points on a circle centered at the origin such that for each point $p \in S$ also its antipodal point $p'$ belongs to $S$. We choose one point from each antipodal pair of $S$ such that their convex hull is as large or as small (w.r.t. its area) as possible. Intuitively speaking, the largest polygon will have to contain the center of the circle, but the smallest one does not. In Figure~\ref{fig:thin_thick} an example of a thin (not containing the center) and a thick (containing the center) polygon is shown. An interesting question, which immediately suggests itself, is whether any thick polygon of $S$ has larger area than any thin polygon of $S$? In this paper, we will formalize the mentioned concepts of thin and thick polygons and answer this question for sets in the plane as well as for higher dimensions.

We start by introducing the problem formally in the plane. The generalization for higher dimensions is straightforward.
A set of $2n$ points on the unit circle centered at the origin
is called an \emph{antipodal point set} if for every point $p$ it also contains its antipodal point $p'$.
Let $S:=\{p_1,p_1',p_2,p_2',\dots,p_n,p_n'\}$ be
such a set. An \emph{antipodal} polygon on $S$
is a convex polygon having as vertices precisely one point
from each antipodal pair $(p_i,p_i')$ of $S$. A \emph{thin} antipodal polygon
is an antipodal polygon whose vertices all lie in a half-plane defined
by some line through the origin.  A \emph{thick} antipodal polygon is
an antipodal polygon such that at least $\left \lceil \frac{n-2}{2}
\right \rceil$ of its vertices lie in both open half-planes defined
by any given line through the origin. See Figure~\ref{fig:thin_thick}.
Note that a non-thin antipodal
polygon does not need to be thick, but a thick antipodal polygon can
never be thin. Moreover, a thin antipodal
polygon does not contain the center of the circle and a non-thin antipodal polygon always contains it.

\begin{figure}
\begin{center}
  \includegraphics[width=0.5\textwidth]{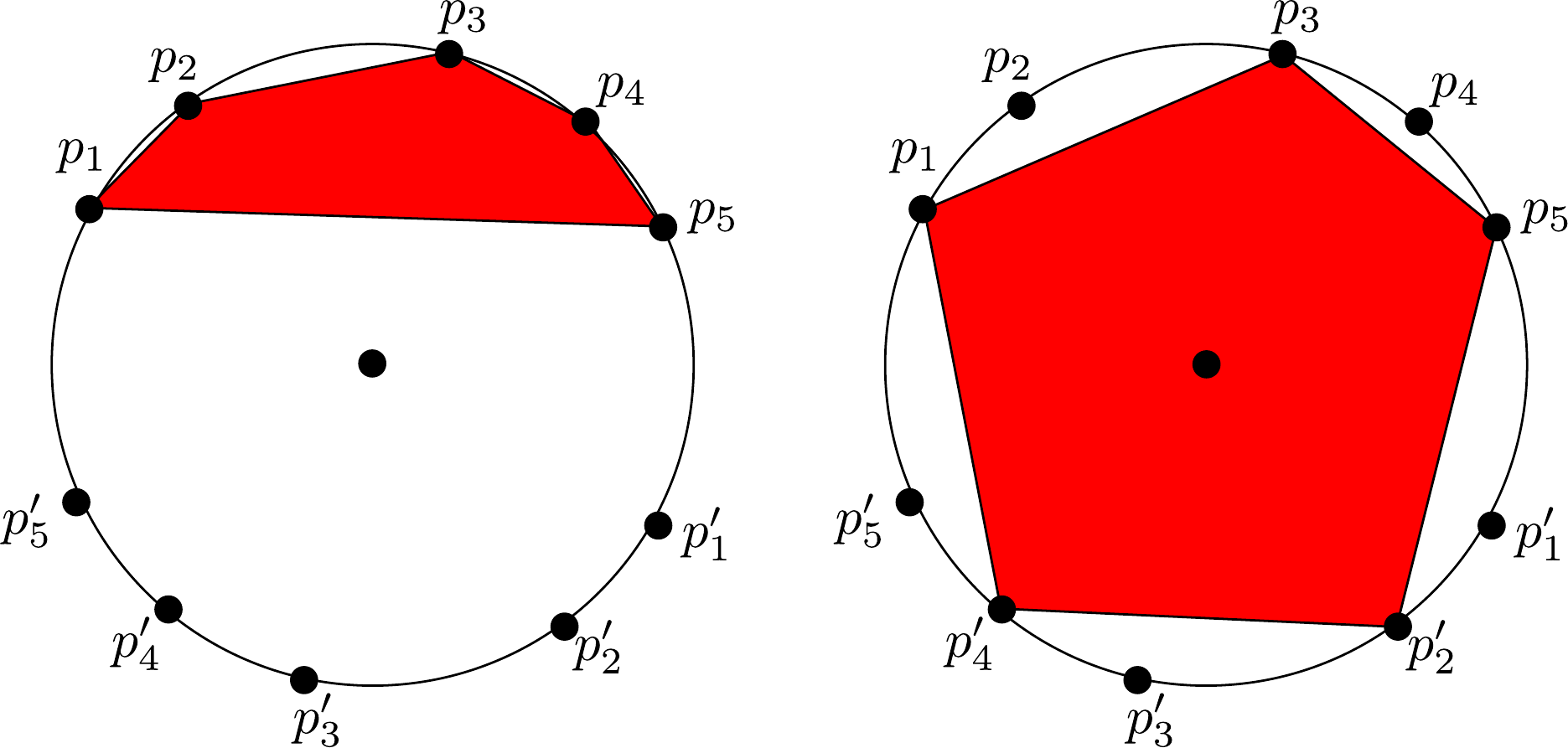}
  \caption{A thin (left) and a thick (right) antipodal polygon.}\label{fig:thin_thick}
\end{center}
\end{figure}

In this paper we investigate the following questions:

\begin{itemize}
\item Does a thick antipodal polygon always have larger area than a thin antipodal polygon?
\item How efficiently can one compute an antipodal polygon with minimal (maximal) area?
\item What can be said about antipodal polygons in higher dimensions?
\end{itemize}

\subsection{Related work}

The questions studied here are related to several other geometric problems, some of which we mention below.

\emph {Extremal problems:} Plane geometry is rich of extremal problems, often dating back till the ancient Greeks. During the centuries many of these problems have been solved by geometrical reasoning. Specifically, extremal problems on convex polygons have attracted the attention of both fields, geometry and optimization. In computational geometry, efficient algorithms have been proposed for computing extremal polygons w.r.t. several different properties~\cite{boyce}. In operations research, global optimization techniques have been extensively studied
 to find convex polygons maximizing a given parameter~\cite{hansen}. A geometric extremal problem similar to the one studied in this paper was solved by Fejes T\'{o}th ~\cite{fejes} almost fifty years ago. He showed that the sum of pairwise distances determined by $n$ points contained in a circle is maximized when the points are the vertices of a regular $n$-gon inscribed in the circle. Recently, the discrete version of this problem has been reviewed in~\cite{toussaint} and problems considering maximal area instead of the sum of inter-point distances have been solved in~\cite{rappaport}.

\emph{Stabbing problems:} The problem of stabbing a set of objects by a polygon (transversal problems in the mathematics literature) has been widely studied. For example, in computational geometry, Arkin et al. \cite{arkin}  considered the following problem: a set $S$ of segments is \emph{stabbable} if there exists a convex polygon whose boundary $C$ intersects every segment in $S$; the
closed convex chain $C$ is then called a (convex) transversal or \emph{stabber} of $S$. Arkin et al. proved that deciding whether $S$ is stabbable is an NP-hard problem.
In a recent paper \cite{pilz}, the problem of stabbing the set $S$ of line segments by a simple polygon but with a different criterion has been considered. A segment $s$ is stabbed by a simple polygon $P$ if at least one of the two endpoints of $s$ is contained in $P$. Then the problem is: Find a
simple polygon $P$ that stabs $S$ and has minimum(maximum) area among those that stab $S$. In \cite{pilz}, it is shown that if $S$ is a set of $n$ pairwise disjoint segments, the problem of computing the minimum and maximum area (perimeter) polygon stabbing $S$ can be solved in polynomial time. However, for general (crossing) segments the problem is APX-hard. Notice  that our problem is  a constrained version of the problem studied  in \cite{pilz} in which each segment joins two antipodal points on a circle. As we will show later, our antipodal version (in which all segments intersect at the origin) can be computed in linear time.

 \begin{figure}
\begin{center}
  \includegraphics[width=0.5\textwidth]{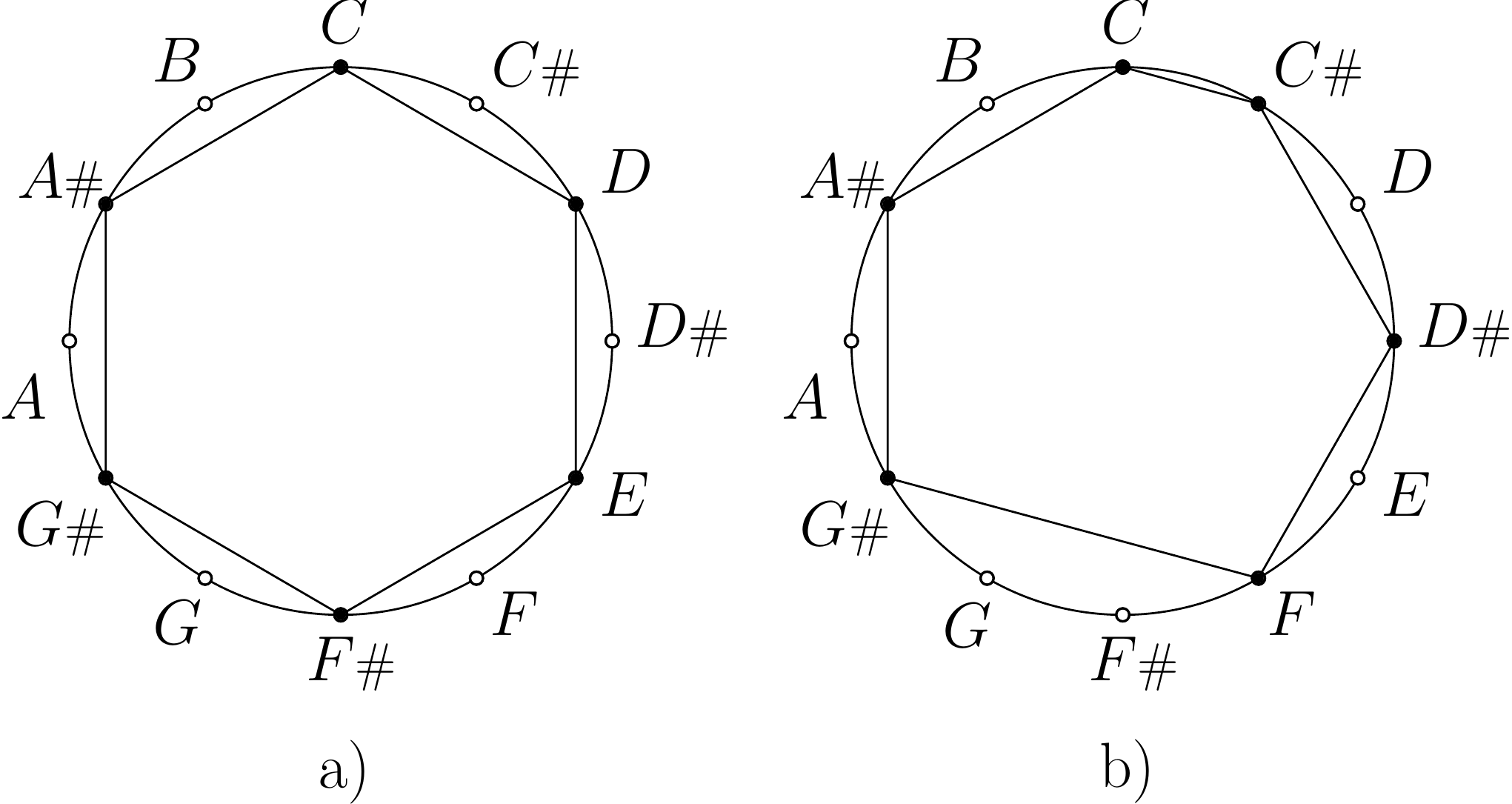}
  \caption{The subsets in a) and b) represent maximally even scales with and without tritones, respectively.}\label{fig:tritone}
\end{center}
\end{figure}

\emph{Music Theory:} There exists a surprisingly high number of applications of mathematics to music theory. Questions about variation, similarity, enumeration, and classification of musical structures have long intrigued both musicians and mathematicians. In some cases, these problems inspired mathematical discoveries. The research in music theory has illuminated problems that are appealing, nontrivial, and, in some cases, connected to deep mathematical questions. See for example \cite{behrends,benson} for introductions to the interplay between mathematics and music.

In our case, an antipodal polygon is related with the \emph{tritone} concept in music theory.  Typically, the notes of a scale are represented by a polygon in a clock diagram. In a chromatic scale, each whole tone can be further divided into two semitones. Thus, we can think in a clock diagram with twelve
points representing the twelve equally spaced pitches that represent the chromatic universe (using
an equal tempered tuning). The pitch class diagram is illustrated in Figure~\ref{fig:tritone} . A tritone is traditionally defined as a musical interval composed of three whole tones. Thus, it is any interval spanning six semitones.
In Figure~\ref{fig:tritone} a), the polygon represents a scale containing the tritones $CF\#, DG\#, EA\#$. The tritone is defined as a restless interval or dissonance in Western music from the early Middle Ages. This interval was frequently avoided in medieval ecclesiastical singing because of its dissonant quality. The name \emph{diabolus in musica} (the Devil in music) has been applied to the interval from at least the early 18th century \cite{grove}.

In this context, an antipodal polygon corresponds to a subset of notes or harmonic scale avoiding the tritone and, according to \cite{rappaport, toussaint},
 a maximal antipodal polygon represents a maximally even set that avoids the tritone.

\subsection{Our results}
In this paper we show that:

\begin{claim}\label{claim:thin}
  For a given antipodal point set $S \in \R^2$ every thin antipodal
  polygon on $S$ has less area than any non-thin antipodal polygon on
  $S$.
\end{claim}

In addition we show that the 2-dimensional case is special in the
sense that the above result can not be generalized to higher
dimensions.

The analogue result holds for thick antipodal
polygons when $n$ is odd but surprisingly turns out to be wrong when $n$ is even; for $n$
even we provide an example of an antipodal non-thick polygon having
larger area than a thick antipodal polygon. However we are able to
show that:

\begin{claim}\label{claim:thick}
  For a given antipodal point set $S \in \R^2$ and every non-thick
  antipodal polygon on $S$, there exists a thick antipodal polygon on
  $S$ with larger area.
\end{claim}

Note that above claims imply that an antipodal
polygon with minimum (resp. maximum) area is thin (resp. thick).

\section{Thin antipodal polygons}\label{sec:thin}

Assume that the clockwise circularly order of $S$ around the origin is
$p_1,p_2,\dots,p_n,$ $p_1',p_2',\dots,p_n'$.
For every point $q$ in $S$, let
$S_q$ be the thin antipodal polygon that contains $q$ as a vertex and
all $n-1$ next consecutive points clockwise from $q$.  Note that all
thin antipodal polygons are of this form and that $S_q$ and $S_q'$
are congruent.

First, we prove a lemma regarding the triangles containing
a given point of $S$.

\begin{lemma}\label{lem:ear}
  For a point $p \in S$ let $\ell$ be the line containing $p$ and $p'$.
  Let $\tau$ be the triangle determined by $p$,
  and its two neighbors in $S$. Among all triangles that have as
  vertices $p$ and one point of $S$ in each of the two half-planes
  defined by $\ell$, $\tau$ has strictly the smallest area.
\end{lemma}
\begin{proof}
  Let $\tau'$ be a triangle with vertices in $S$, containing $p$ as a
  vertex and with a vertex in each of the two half-planes defined by
  $\ell$. Assume that $\tau'$ is different from~$\tau$.  Let $b$ be
  the side opposite to $p$ in $\tau$ and $b'$ be the side opposite to
  $p$ in $\tau'$.  Note that $b'$ is at least as large as $b$,
  because $S$ is an antipodal point set and $\ell$ contains
  the origin. The height of $\tau'$ with
  respect to $p$ is greater than the height of $\tau$ with respect to
  $p$, as otherwise $b'$ would have to intersect $b$, which is not
  possible by construction. Thus the area of $\tau'$ is larger than the
  area of $\tau$.
\end{proof}

%
%
%

We split the proof of Claim~\ref{claim:thin} into the three cases
$n=3$, $n=4$, and $n \geq 5$.

\begin{lemma}\label{lem:n=3}
  For $n=3$, every thin
  antipodal polygon on $S$ has an area strictly less than that of any
  non-thin antipodal polygon on $S$.
\end{lemma}

%

\begin{proof}
In this case the only non-thin polygons are the two triangles $\tau$ and $\tau'$ with vertex sets $\{p_1,p_2',p_3\}$ and $\{p_1',p_2,p_3'\}$, respectively. Note that $\tau$ has the same area as $\tau'$. In addition, by Lemma \ref{lem:ear}$,  \tau$ has greater area than $S_{p_2}$ and $\tau^\prime$ has greater area than $S_{p_1}$ and $S_{p_3}$.
\end{proof}

\begin{figure}
\begin{center}
  \includegraphics[width=0.5\textwidth]{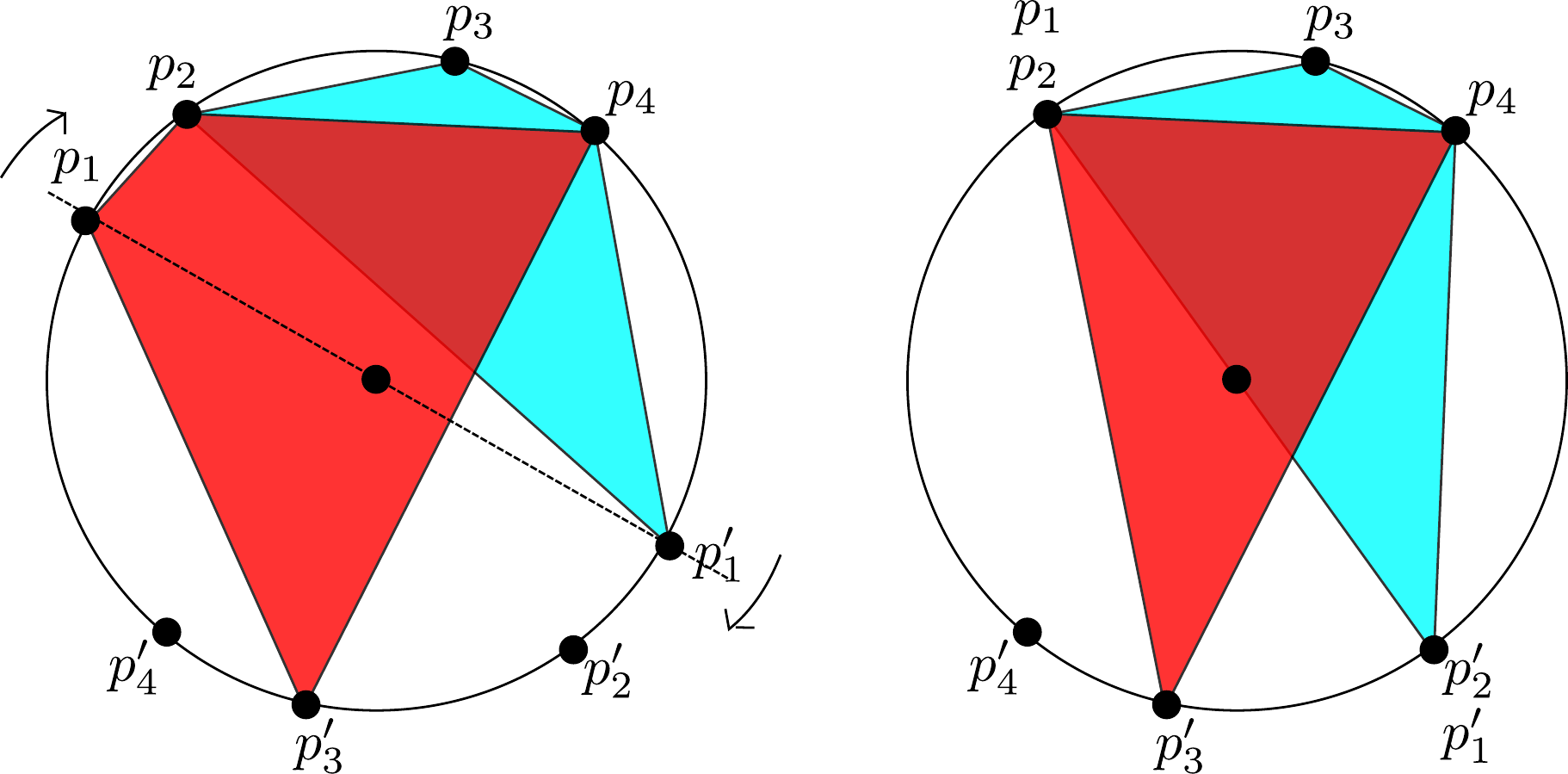}
  \caption{The rotation in the proof of Lemma~\ref{lem:n=4} and its
limit case.}\label{fig:n_4}
\end{center}
\end{figure}

\begin{lemma}\label{lem:n=4}
  For $n=4$, every thin
  antipodal polygon on $S$ has an area strictly less than that of any
  non-thin antipodal polygon on $S$.
\end{lemma}
\begin{proof}
  In this case a non-thin antipodal polygon $P$ has exactly two
  consecutive points; without loss of generality assume that they are
  $p_1$ and $p_2$.  Thus $P$ is the convex quadrilateral
  $p_1,p_2,p_4,p_3'$. We show that $P$ has greater area than
  $S_{p_1},S_{p_2},S_{p_3'}$ and $S_{p_4'}$.

  By Lemma~\ref{lem:ear} the triangle $p_4'p_1p_2$ has less
  area than the triangle $p_3'p_1p_2$. By Lemma~\ref{lem:n=3} the
  triangle $p_3'p_2p_4$ has an area greater than the triangle
  $p_3'p_4'p_2$ and also greater than the triangle $p_4'p_2p_3$. Thus
  $P$ has an area greater than $S_{p_3'}$ and also greater than
  $S_{p_4'}$.  By Lemma~\ref{lem:ear} the triangle $p_1p_2p_3$ has
  less area than the triangle $p_1p_2p_4$. By Lemma~\ref{lem:n=3} the
  triangle $p_3'p_1p_4$ has an area greater than the triangle
  $p_1p_3p_4$.  Thus $P$ has an area greater than $S_{p_1}$.

It remains to show that $P$ has area greater than $S_{p_2}$.
Let $\ell$ be the line passing through $p_1$ and $p_1'$.
Rotate $\ell$ clockwise continuously around the origin,
until $p_1$ meets $p_2$ and $p_1'$ meets $p_2'$. See Figure~\ref{fig:n_4}.
Note that throughout the motion the area of $S_{p_2}$ is
strictly increasing. To see that, notice that the height of the triangle with vertices $p_2,p_4$ and
 $p_1$ is strictly increasing, as
otherwise, at some point $p_1'$ must intersect the perpendicular
bisector of the segment
$p_2p_4$. However, this cannot happen since $p_1'$ reaches $p_2'$  before it
reaches this line.

On the other hand, the area of $P$ might at first
be strictly increasing, then at some point be strictly decreasing.
Moreover, if this is the case, there is a point in time, at which
$P$ has the same area as in the beginning of the motion
(and will strictly decrease afterwards) and the area
of $S_{p_2}$ has increased. Assume then that
the motion is such that the area of $P$ is strictly decreasing
and the area of $S_{p_2}$ is strictly increasing.

We show
that at the end of the motion $P$ and $S_{p_2}$ have
equal area, this implies that at the beginning of the motion
the area of $P$ is greater than the area of $S_{p_2}$.

 At the end of the motion $P$ coincides with the triangle $p_2p_4p_3'$ and $S_{p_2}$ with the quadrilateral $p_2p_3p_4p_2'$. We split the the quadrilateral $p_2p_3p_4p_2'$ into the triangles $p_2p_3p_4$ and $p_2^\prime p_2 p_4$, sharing the side $\overline{p_2 p_4}$. The height of the triangle $p_2p_4p_3'$ with respect to $\overline{p_2 p_4}$ has the same length that the sum of the heights of the triangles $p_2p_3p_4$ and $p_2^\prime p_2 p_4$ with respect to $\overline{p_2 p_4}$ (It is easy to see by using the triangle $p_4'p_3'p_2'$). Hence $\operatorname{Area}(p_2p_4p_3')$ equals $\operatorname{Area}(p_2p_3p_4p_2')$. \end{proof}

We are ready now to prove our first claim.

\begin{theorem}\label{thm:thin}
Every thin antipodal polygon on $S$ has less area than
any non-thin antipodal polygon on $S$.
\end{theorem}
\begin{proof}
  We proceed by induction on $n$. By Lemmas \ref{lem:n=3} and
  \ref{lem:n=4}, we assume that $n\ge 5$.  Let $P$ be a non-thin
  antipodal polygon on $S$. Let $T$ be any triangulation of $P$. Let
  $p$ be a vertex of degree two in $T$ and let $p'$ be its antipodal
  point.  Let $\tau$ be the only triangle of $T$ having $p$ as a
  vertex.  Let $q$ and $r$ be the two neighbors of $p$ in $S$.
  Let $\tau' $ be the triangle with vertices
  $p$, $q$ and $r$. By Lemma~\ref{lem:ear} the area of $\tau'$ is
  equal or less than the area of $\tau$.

Now, suppose that $\tau$ does not contain the origin
in its interior, then the polygon $P'$ with vertices
$V(P)\setminus\{ p \}$ is a non-thin antipodal polygon
for $S\setminus\{p,p'\}$. By induction $P'$ has area greater
area than any thin antipodal polygon on $S\setminus\{p,p'\}$.
Some of these thin polygons together with $\tau'$ form antipodal polygons
on $S$. Using this observation and the fact that the area of $S_{p_i}$
is the same as the area of $S_{p_i'}$, we can show that
except for $S_{p}$ and $S_{q}$ all antipodal thin polygons
on $S$ have area strictly less than $P$. However,
for $n\ge 5$, $P$ can be triangulated so that $p$ is not
the middle nor the last vertex (clockwise) of an ear.
As any triangulation has two ears. There is an ear that does
not contain the origin. The previous arguments (for this ear)
show that the area of $P$ is strictly greater than
the area of $S_p$, similarly for $S_q$.
\end{proof}

\section{Thick antipodal polygons}\label{sec:thick}

In this section we present two area increasing operations on antipodal
polygons. Using a sequence of these operations a non-thick
antipodal polygon can be transformed into a thick antipodal
polygon, this sequence proves Theorem~\ref{thm:thick}.

We begin with an antipodal polygon $P$.
Let $q$ be a point in $S$. By \emph{flipping} $q$, we mean
the following operation: if $q$ is a vertex
of $P$, then choose $q'$ instead; if
$q$ is not a vertex of $P$ then choose
$q$ instead of $q'$.
The two operations described
in Lemmas~\ref{lem:3cons} and \ref{lem:2cons} are sequences of such flips.

\begin{lemma}\label{lem:3cons}
If $P$ has three consecutive points
 $q_1,q_2$ and $q_3$ of $S$ as vertices,
then flipping $q_2$, provides a polygon
$P$ of greater area.
\end{lemma}
\begin{proof}
Let $q_4'$ be the point after $q_3'$ in $P$ and $q_0'$ be the point
before $q_1'$ in $P$. Let $\tau_1$ be the triangle with vertex set
$\{q_1,q_2,q_3\}$ and $\tau_2$ the triangle with vertex set
$\{q_0',q_2',q_4'\}$.
The difference of the areas of $P$ and $P'$ is equal to the difference
in the areas of $\tau_1$ and $\tau_2$. However, $\tau_1$ has the same
area as the triangle with vertex set $\{q_1',q_2',q_3'\}$;
by Lemma~\ref{lem:ear} the area of this triangle is
less than that of~$\tau_2$.
\end{proof}

From now on, we assume that $P$ does not contain three consecutive points of
$S$ as vertices. Otherwise we apply the operation described
in Lemma~\ref{lem:3cons}.

\begin{figure}
\begin{center}
  \includegraphics[width=0.75\textwidth]{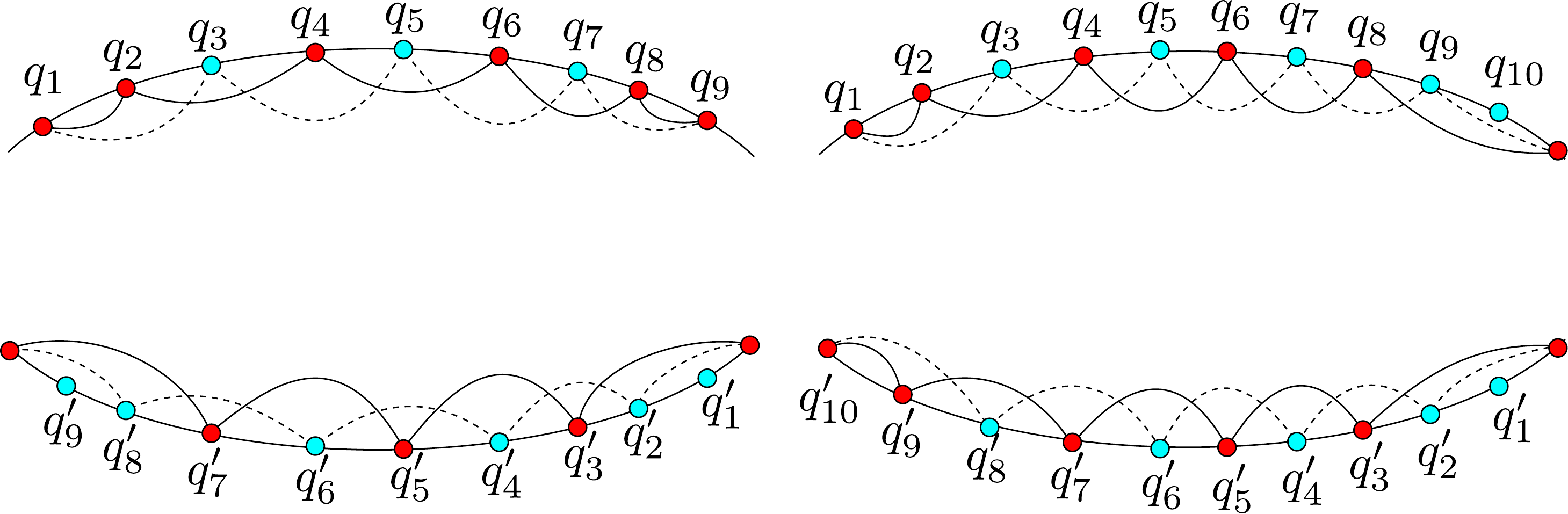}
  \caption{Schematic diagram of the two flip operations
    described in Lemma~\ref{lem:2cons}.
    $P$ is drawn solid and $P'$ is dashed.}\label{fig:flips}
\end{center}
\end{figure}

\begin{lemma}\label{lem:2cons}
 Let $q_1,q_2,\dots,q_m$ $(4\le m <n)$ be consecutive
points of $S$. Suppose that:
\begin{itemize}
\item  $P$ contains $q_1$ and $q_2$.
\item $P$ contains either both $q_{m-1}$ and $q_m$,
or neither of them.
\item The points from $q_2$ to $q_{m-1}$ alternatingly belong to $P$ or not.
\end{itemize}
Let $P'$ be the antipodal polygon obtained from $P$,
by flipping each point $q_i$ $(2\le i \le {m-1})$.
Then $P'$ has greater area than $P$.
\end{lemma}
\begin{proof}
  Note that $\mathcal{T}:=(P \setminus P')\cup(P'\setminus P)$ is a set of interior 
  disjoint triangles. For each $p$ in $\{q_2,q_2',\dots,q_{m-1},q_{m-1}'\}$ let 
  $\tau(p)$ be the triangle in $\mathcal{T}$ that contains p as a vertex.
  The difference in the area of $P$ and the area of $P'$ equals the difference in the
  areas of those triangles contained in $P$ and those contained in $P'$.
  For $4\le i \le m-3$,
  the area of $\tau(q_i)$ equals the area of
  $\tau(q_i')$ and one of them is contained in $P$ while the other is
  contained in $P'$.  Thus the difference in the areas of $P$ and $P'$
  depends only on the areas of $\tau(q_2)$, $\tau(q_2')$,  $\tau(q_3)$, $\tau(q_3')$, $\tau(q_{m-2})$,
  $\tau(q_{m-2}')$, $\tau(q_{m-1})$,
  and $\tau(q_{m-1}')$
 Note that the area of $\tau(q_2)$ is smaller  than the
  area of $\tau(q_2')$ and that $P$ contains $\tau(q_2)$ while $P'$ contains
  $\tau(q_2')$. Similarly for $\tau(q_3)$ and $\tau(q_3')$). See Figure~\ref{fig:flips}.

If $P$ contains both $q_{m-1}$ and $q_{m}$, then
$\tau(q_{m-1})$ is contained in $P$ and
$\tau(q_{m-1}')$ is contained in $P'$.
In this case the area
of $\tau(q_{m-1})$ is smaller than the area
of $\tau(q_{m-1}')$.

If $P$ does not contain $q_{m-1}$ and $q_{m}$, then
$\tau(q_{m-1}')$ is contained in $P$ and
$\tau(q_{m-1})$ is contained in $P'$.
In this case the area
of $\tau(q_{m-1}')$ is smaller than the area
of $\tau(q_{m-1})$. The same argument can by apply to $\tau(q_{m-2})$ and $\tau(q_{m-2}')$).
Thus, in all cases the area of $P$ is smaller
than the area of $P'$.
\end{proof}

Note that in the operation described in
Lemma~\ref{lem:2cons} the number of pairs of consecutive
points that are either both on $P$ or not in $P$ decreases.
Moreover, no three consecutive points
all in $P$ or all not in $P$ are created at the same time .

We are now ready to prove the second claim.

\begin{theorem}\label{thm:thick}
For every non-thick antipodal polygon on $S$, there exists
a thick antipodal polygon on $S$ of greater area.
\end{theorem}
\begin{proof}


For $n$ odd, an antipodal polygon $Q$ is thick if and only if
its points alternate
between being in $Q$ and not being in $Q$. For $n$ even, an antipodal polygon
is thick if and only if its points alternate between being in $Q$ and
not in $Q$, with
the exception of exactly one pair of consecutive points which are both
in $Q$ (and
its antipodal points not in $Q$).

Assume that all possible operations of
Lemmas~\ref{lem:3cons} and~\ref{lem:2cons} have been applied to a
non-thick antipodal polygon $P$, then $P$ contains at most
one pair of consecutive points in $S$ as vertices
and $P$ is a thick polygon.
\end{proof}

\begin{corollary}
For $n$ odd, every thick antipodal polygon on $S$
has greater area than a non-thick antipodal polygon on $S$.
\end{corollary}
\begin{proof}
In this case there are only two antipodal thick polygons
and they have the same area.
\end{proof}

We now provide an example of a set of points and a non-thick antipodal polygon
that has greater area than a thick antipodal polygon on this set.

\begin{theorem}
For $n \geq 6$ even, there exist point sets with a non-thick antipodal
polygon of greater area than a thick antipodal polygon.
\end{theorem}
\begin{proof}
Place $p_1$ and $p_2$ arbitrarily close to $(1,0)$;
 thus $p_1'$ and $p_2'$ are arbitrarily close to
$(-1,0)$. Place $p_3,\dots, p_n$ arbitrarily close to
$(0,1)$; thus $p_3', \dots, p_n'$ are arbitrarily close
to $(0,-1)$. Let $P$ be the thick antipodal polygon
that contains both $p_1$ and $p_2$ as vertices.
Let $Q$ be any non-thick antipodal polygon that contains
$p_1,p_2',p_3$ and $p_4'$ as vertices. Note that $P$ is arbitrarily
close to the triangle with vertices $(0,1)$, $(0,-1)$ and
$(1,0)$; $Q$ is arbitrarily close to the quadrilateral with
vertices $(-1,0)$, $(0,1)$, $(1,0)$, and $(0,-1)$.
Thus the area of $P$ is arbitrarily close to $1$, while
the area of $Q$ is arbitrarily close to $2$.
\end{proof}

\section{The algorithms}

 It is worth mentioning that the general algorithmic version of the problem in which the input is a set of line segments, each connecting two points on the circle, has been proved to be NP-hard \cite{pilz}. Surprisingly, the antipodal version can be easily solved by using above characterizations.

\begin{theorem}
Antipodal polygons with minimum or maximum the area can be computed in linear time.
\end{theorem}
\begin{proof}
According to Theorem \ref{thm:thin}, an antipodal polygon with minimum area is a thin antipodal polygon. Thus, since there exist $O(n)$ thin polygons, we can sweep in a linear number of steps around the circle and update in constant time the area of two consecutive thin polygons.
On the other hand, according to Theorem \ref{thm:thick}, if $n$ is odd, there are only two thick antipodal polygons (the alternating polygons). For $n$ even, there exists a linear number of thick polygons (having two consecutive points and the rest in alternating position). In the last case, a linear sweep around the circle can also be used to compute in linear time a thick antipodal polygon that maximizes the area.
\end{proof}

\section{Higher Dimensions: Antipodal Polytopes}

In this section we consider the analogous problem in higher
dimensions.  Assume therefore that all points are now placed on the
unit $d$-dimensional sphere.  Instead of antipodal polygons we thus
have antipodal polytopes. For a thin antipodal polytope all its points
lie on one side of some hyperplane passing through the origin.

In dimension $3$ or greater Theorem~\ref{thm:thin} does not
hold---there are antipodal point sets $S \subset \R^d$ such that there
exists an antipodal thin polytope with greater $d$-dimensional volume
than a non-thin antipodal polytope on $S$.  We start by providing a
three dimensional example and then argue how to generalize it to
higher dimensions.

For some small $\varepsilon > 0$, let $\delta =
\sqrt{1-2\varepsilon^2}$ and consider the set $S_1$ of the five points
$v_1:=(0, 0, 1)$, $v_2:=(\delta, \varepsilon, \varepsilon)$,
$v_3:=(-\delta, \varepsilon, \varepsilon)$, $v_4:=(\varepsilon,
\delta, \varepsilon)$, and $v_5:=(\varepsilon, -\delta, \varepsilon)$.
Let $S$ be the antipodal point set consisting of $S_1$ and all
its antipodal points.  The convex hull of $S_1$ is
a pyramid with a square base (with corners $v_2 ,\dots, v_5$) which
lies in the horizontal plane just $\varepsilon$ above the origin. The
top of the pyramid is at height 1. Thus, this pyramid does not contain
the origin in its interior, and for $\varepsilon \rightarrow 0$ the
volume of the pyramid converges to $2/3$.

To obtain our second polyhedra first flip the vertex $v_1$ to $v_1':=
(0, 0, -1)$. This gives a similar upside-down pyramid, which contains
the origin in its interior. By also flipping $v_2$ to $v_2':=
(-\delta, -\varepsilon, -\varepsilon)$, we essentially halve the base
of the pyramid to be a triangle.  We denote the resulting point set by
$S_2 = \{v_1', v_2', v_3 , v_4 , v_5 \} \subset S$. Note that $v_2'$
and $v_3$ are rather close together. As the triangle $v_3 , v_4 , v_5
$ lies above the origin, the convex hull of $S_2$ still contains the
origin in its interior. Moreover, the volume of the convex hull of
$S_2$ converges to $1/3$ for $\varepsilon \rightarrow 0$, and thus
towards half of the volume of the convex hull of $S_1$.

So together these two polyhedra constitute an example which shows that
Theorem~\ref{thm:thin} can not be generalized to higher dimensions:
$S$ is a set of five antipodal pairs of points on the surface of the
3-dimensional unit sphere such that the convex hull of $S_1$ does not
contain the origin, while the convex hull of $S_2$ does. But in the
limit the volume of the convex hull of $S_1$ becomes twice as large as
the volume of the convex hull of $S_2$.

It is straight forward to observe that this example can be generalized
to any dimension $d \geq 4$. There we have $2d-1$ antipodal pairs of
points, where we set $\delta = \sqrt{1-(d-1)\varepsilon^2}$ and every
point has one coordinate at $\pm \delta$ and the remaining coordinates
at $\pm \varepsilon$, analogous to the 3-dimensional case.  For $d-1$
of the coordinate axes two such pairs are 'aligned' as in the
3-dimensional example, and for the last axis there is only one such
pair. The resulting polytope does not contain the origin.  Flipping
the vertex of the singular pair and one vertex for all but one aligned
pairs results in a polytope which contains
the origin, but has a volume of only $1/2^{d-2}$ of the first polytope.\\


We call a $d$-dimensional antipodal polytope \emph{thick} if the number of
vertices in any half-space defined by a hyperplane through the origin
contains at least $\left \lceil \frac{n-d}{2} \right \rceil$ points of the
polytope. Note that this definition generalizes the two dimensional
case.

It is not clear that for a given antipodal set in $\R^d$ an antipodal
thick polytope should exist. However, for every $n\ge d$, there
exists antipodal sets in $\R^d$ that admit an antipodal
thick polytope. We use the following Lemma.

\begin{lemma} {\bf (Gale's Lemma \cite{gale})}. For every
$d \ge 0$ and every $k \ge 1$, there exists a set $X \subset S^d$
of $2k+d$ points such that every open hemisphere of $S^d$ contains
at least $k$ points of $X$.
\end{lemma}

From the proof of Gale's Lemma in \cite{mat} (page 64), it
follows that the provided set does not contain an antipodal pair of points.
Recall that $S^{d-1} \subset \R^d$;
let $X$ be the subset of $S^{d-1}$ provided by Gale's Lemma for
$k=\left \lceil \frac{n-d+1}{2} \right \rceil$. If necessary
remove a point from $X$ so that $X$ consists of exactly $n$ points.
Let $X'$  be the set of antipodal points of $X$. Set $S:=X\cup X'$.
Let $P$ be the antipodal polytope on $S$ with $X$ as a vertex set.
It follows from Gale's Lemma that $P$
is thick.

\section{Open problems}

Let us assume that we are given a circular lattice with an antipodal set of $2n$ points (evenly spaced) and we would like to compute an extremal antipodal $k$-polygon with $k<n$ vertices. This problem is significantly different to the considered case $k=n$. Recall that, for $k=n$, the linear algorithms proposed in this paper are strongly based on the simple characterization for the extremal antipodal polygons. Namely, the minimal thin antipodal polygon has consecutive vertices and the thick one has an alternating configuration.  It is not difficult to come up with examples for which that characterization does not hold in the general case $k<n$. On the other hand, finding the extremal antipodal $(n-1)$-polygon, called $(2n,n-1)$-problem for short, can be easily reduced to solve $O(n)$ times the $(2(n-1),n-1)$-problem. To see this, observe that in the $(2n,n-1)$-problem an antipodal pair is not selected and can thus be removed from the input.
This approach gives a simple $O(n^{n-k+1})$ time algorithm for solving the general $(2n,k)$-problem. This leaves as open problem to prove if the $(2n,k)$-problem can be solve in $o(n^k)$ time.

Instead of area, it is also interesting to consider other extremal measures, like perimeter or the sum of inter-point distances.
Finally, for higher dimensions, we leave the existence of thick polytopes for arbitrary antipodal
point sets as an open problem.

\section{Acknowledgments}

The problems studied here were introduced and partially solved during a visit
to University of La Havana, Cuba.

\bibliographystyle{plain}
\bibliography{antipodalbib}

\end{document}